\documentclass[a4paper]{amsart}
\usepackage[utf8]{inputenc}
\usepackage{tikz}

\title[Invariant measures on products and on LO]{Invariant measures on products and on the space of linear orders}

\author{Colin Jahel}
\address{
  Institut Camille Jordan \\
  Universit\'e Claude Bernard Lyon 1 \\
  Universit\'e de Lyon \\  
  43, boulevard du 11 novembre 1918 \\
  69622 Villeurbanne \textsc{cedex} \\
  France}
\email{colin.jahel@ens-lyon.fr}

\author{Todor Tsankov}
\address{
  Institut Camille Jordan \\
  Universit\'e Claude Bernard Lyon 1 \\
  Universit\'e de Lyon \\
  43, boulevard du 11 novembre 1918 \\
  69622 Villeurbanne \textsc{cedex} \\
  France
  -- and --
  Institut Universitaire de France}
\email{tsankov@math.univ-lyon1.fr}

\makeatletter
\@namedef{subjclassname@2020}{\textup{2020} Mathematics Subject Classification}
\makeatother


\usepackage[T1]{fontenc}
\usepackage[osf]{mathpazo}
\usepackage{amssymb} 

\usepackage{eucal}

\usepackage{hyperref}
\usepackage{my-macros}

\usepackage{enumitem}
\setlist[enumerate,1]{label=(\roman*), font=\normalfont}

\usepackage[initials, shortalphabetic]{amsrefs}

\newcommand{\muu}{\mu_{\mathrm{u}}}
\newcommand{\Perp}[1][]{\mathrel{
    \mathop{
      \scalebox{1.3}{\text{$\perp$}}
      }\displaylimits_{#1}}}

\begin{document}

\begin{abstract}
  Let $M$ be an $\aleph_0$-categorical structure and assume that $M$ has no algebraicity and has weak elimination of imaginaries. Generalizing classical theorems of de Finetti and Ryll-Nardzewski, we show that any ergodic, $\operatorname{Aut}(M)$-invariant measure on $[0, 1]^M$ is a product measure. We also investigate the action of $\operatorname{Aut}(M)$ on the compact space $\mathrm{LO}(M)$ of linear orders on $M$. If we assume moreover that the action $\operatorname{Aut}(M) \curvearrowright M$ is transitive, we prove that the action $\operatorname{Aut}(M) \curvearrowright \mathrm{LO}(M)$ either has a fixed point or is uniquely ergodic.
\end{abstract}

\maketitle

\section{Introduction}
\label{sec:introduction}

In recent years, the study of dynamical systems of automorphism groups of homogeneous structures has become an important topic at the intersection of dynamics, combinatorics, probability theory, and model theory and it has uncovered many interesting connections between these fields. Countable homogeneous structures are obtained as \emph{Fraïssé limits} of a class of finite structures satisfying certain conditions (called a \df{Fraïssé class}) and there is a close correspondence between dynamical properties of the automorphism group of the limit structure and combinatorial properties of the class. Typical examples of Fraïssé classes are the class of finite graphs (the limit is the random graph), the class of finite triangle-free graphs, and the class of finite linear orders (here the limit is the countable, dense linear order without endpoints $(\Q, <)$).

In this paper, we will be interested in the invariant probability measures on dynamical systems of the automorphism group $\Aut(M)$ of a homogeneous structure $M$. More precisely, we will consider two specific systems: products of the type $Z^M$, where $Z$ is a standard Borel space, and the compact space $\LO(M)$ of all linear orders on $M$.

Our study of invariant measures on product spaces of the type $Z^M$ is inspired by the classical de Finetti theorem. One formulation of this theorem is that the only ergodic measures on $Z^M$ invariant under the full symmetric group $\Sym(M)$ are product measures of the type $\lambda^M$, where $\lambda$ is some probability measure on $Z$. (Recall that a measure is \df{ergodic} if the only elements of the measure algebra fixed by the group are the empty set and the whole space. Equivalently, the ergodic measures are the extreme points of the convex set of all invariant probability measures.) In our first result, we obtain the same conclusion as in de Finetti's theorem under a weaker hypothesis: that the measure is invariant under the much smaller group $\Aut(M)$, provided that the structure $M$ satisfies certain model-theoretic conditions. We will say that a structure $M$ is \df{transitive} if the action $\Aut(M) \actson M$ is transitive.

\begin{theorem}
  \label{th:intro:deFinetti}
  Let $M$ be an $\aleph_0$-categorical, transitive structure with no algebraicity that admits weak elimination of imaginaries. Let $Z$ be a standard Borel space and consider the natural action $\Aut(M) \actson Z^M$. Then the only invariant, ergodic probability measures on $Z^M$ are product measures of the form $\lambda^M$, where $\lambda$ is a probability measure on $Z$.
\end{theorem}
We will discuss the model-theoretic hypotheses of the theorem in detail in the next section, where we give all relevant definitions. The proof can be found in Section~\ref{sec:general-de-finett} (cf. Corollary~\ref{c:deFinetti}). Here we only remark that they are all necessary (with the possible exception of $\aleph_0$-categoricity) and that they are satisfied, for example, by the random graph, the homogeneous triangle-free graph, the dense linear order $(\Q, <)$, the universal, homogeneous partial order, and many other structures.

The ergodicity assumption in the theorem entails no loss of generality: one can obtain a description of all invariant measures using the ergodic decomposition theorem.


Theorem~\ref{th:intro:deFinetti} is a consequence of a more general independence result that applies to any measure-preserving action of $\Aut(M)$ on a probability space for any $\aleph_0$-categorical structure $M$ (cf. Theorem~\ref{th:general-independence}). The proof is based on representation theory and the results of \cite{Tsankov2012}.

Making use of Fraïssé's theorem, it is also possible to apply Theorem~\ref{th:intro:deFinetti} even in situations where there is no homogeneity or an obvious group present. For example, we can recover a theorem of Ryll-Nardzewski~\cite{RyllNardzewski1957}, which is another well-known strengthening of de Finetti's theorem; cf. Corollary~\ref{c:RyllNardzewski}.

Theorem~\ref{th:intro:deFinetti} was announced in the habilitation memoir of the second author \cite{Tsankov2014}. Later, some independent related work has been done by Ackerman~\cite{Ackerman2015p} and Crane--Towsner~\cite{Crane2018}. They consider a different class of homogeneous structures (with combinatorial assumptions on the amalgamation) and use completely different methods.

Next we consider $\Aut(M)$-invariant probability measures on the compact space $\LO(M)$ of linear orders on $M$. The systematic study of these measures was initiated by Angel, Kechris, and Lyons in  \cite{Angel2014}. Their main motivation comes from abstract topological dynamics. If $G$ is a topological group, a \df{$G$-flow} is a continuous action of $G$ on a compact Hausdorff space. A flow is \df{minimal} if every orbit is dense. It turns out that for every group $G$, there is a \df{universal minimal flow (UMF)} that maps onto every minimal flow of the group. In many cases (for example if $G$ is locally compact, non-compact), the universal minimal flow is a large, non-metrizable space that does not admit a concrete description; however, for many automorphism groups of homogeneous structures $M$, the UMF of $\Aut(M)$ is metrizable and can be explicitly computed. Moreover, in most known examples, it is a subflow of the flow $\LO(M)$ of all linear orders on $M$. In these situations, classifying the invariant measures on $\LO(M)$ gives information about \emph{all} minimal flows of the group as well as other properties of $G$ that can be expressed dynamically. One such property is amenability: a topological group $G$ is called \df{amenable} if every $G$-flow carries an invariant measure, or equivalently, if the UMF of $G$ has an invariant measure. Another property that goes down to factors of the UMF is unique ergodicity: if the UMF is uniquely ergodic, then so is every other minimal flow of the group. (Recall that a flow is \df{uniquely ergodic} if it carries a unique invariant measure, which then must be ergodic.) This latter property is quite interesting and is not encountered in classical dynamics: for example, Weiss~\cite{Weiss2012} has constructed, for every countable, infinite, discrete group, a minimal flow which is not uniquely ergodic and a similar construction was carried out in \cite{Jahel2020p} for locally compact second countable groups.

Giving interesting examples of groups with this unique ergodicity property was one of the main motivations of \cite{Angel2014}. They reduce the unique ergodicity problem to an equivalent question about finite structures (as is often done with Fraïssé limits) and then use techniques from probability theory to attack each specific case. It is interesting that their approach works in both directions, so if one manages to obtain an unique ergodicity results by other methods, this yields combinatorial information about the corresponding Fraïssé class. For example, if we denote by $R$ the random graph, the unique ergodicity of the flow $\Aut(R) \actson \LO(R)$ is equivalent to the uniqueness of a \df{consistent random ordering} on the class of finite graphs (see \cite{Angel2014} for more details). The work of Angel, Kechris, and Lyons was followed by several papers \cites{Pawliuk2020, Jahel2019p}, in which more unique ergodicity results of this type were proved; in particular, the automorphism groups of all homogeneous directed graphs from Cherlin's classification were treated in these two articles.

In the present paper, we adopt a different approach to the unique ergodicity problem on the space of linear orders, based on the generalization of de Finetti's theorem that we discussed above. It has the advantage of working under rather general model-theoretic assumptions (which are mostly necessary) and can also give information about the invariant measures even in the absence of unique ergodicity. Our main theorem is the following.
\begin{theorem}
  \label{th:intro:main} Let $M$ be a transitive, $\aleph_0$-categorical structure with no algebraicity that admits weak elimination of imaginaries. Consider the action $\Aut(M) \actson \LO(M)$. Then exactly one of the following holds:
  \begin{enumerate}
  \item The action $\Aut(M) \actson \LO(M)$ has a fixed point (i.e., there is a definable linear order on $M$);
  \item The action $\Aut(M) \actson \LO(M)$ is uniquely ergodic.
  \end{enumerate}
\end{theorem}
Theorem~\ref{th:intro:main} recovers almost all known results about unique ergodicity of $\LO(M)$. More specifically, it applies to the following structures:
\begin{itemize}
\item the random graph, the $K_n$-free homogeneous graphs, various homogeneous hypergraphs, and the universal homogeneous tournament \cite{Angel2014};
  
\item the generic directed graphs obtained by omitting a (possibly infinite) set of tournaments or a fixed, finite, discrete graph \cite{Pawliuk2020}. 
\end{itemize}
The class of structures satisfying the hypothesis of Theorem~\ref{th:intro:main} is quite a bit richer than the examples above. We should mention, however, that it does not cover all cases where unique ergodicity of the space of linear orders is known. One exception is the rational Urysohn space $\bU_0$: it was proved in \cite{Angel2014} that the action $\Iso(\bU_0) \actson \LO(\bU_0)$ is uniquely ergodic but $\bU_0$ is not $\aleph_0$-categorical (as it has infinitely many $2$-types). It also does not apply directly to prove unique ergodicity for proper subflows of $\LO$, for example for the automorphism group of the countable-dimensional vector space over a finite field.

The proof of Theorem~\ref{th:intro:main} is the object of Section~\ref{sec:invar-meas-space}, where it is stated as Theorem~\ref{Thm:Princip}.

We also have an interesting corollary of Theorem~\ref{th:intro:main} concerning amenability.
\begin{cor}
  \label{c:intro:amenability}
  Suppose that $M$ satisfies the assumptions of Theorem~\ref{th:intro:main}. If the action $\Aut(M) \actson \LO(M)$ is not minimal and has no fixed points, then $\Aut(M)$ is not amenable.
\end{cor}

Corollary~\ref{c:intro:amenability} applies for example to the automorphism groups of the universal homogeneous partial order and the circular directed graphs $\bS(n)$ for $n \geq 2$, recovering results of Kechris--Soki\'c~\cite{Kechris2012} and Zucker~\cite{Zucker2014}, respectively.

Corollary~\ref{c:intro:amenability} also has an interesting purely combinatorial consequence of which we do not know a combinatorial proof. Recall that a Fraïssé class $\cF$ (or its Fraïssé limit) has the \df{Hrushovski property} if partial automorphisms of elements of $\cF$ extend to full automorphisms of superstructures in $\cF$. It has the \df{ordering property} if for every $A \in \cF$, there exists $B \in \cF$ such that for any two linear orders $<$ and $<'$ on $A$ and $B$ respectively, there is an embedding of $(A, <)$ into $(B, <')$. The Hrushovski and the ordering properties are important in the theory of homogeneous structures and in structural Ramsey theory but are not a priori related. We refer the reader to \cite{Kechris2007a} and \cite{Nesetril1978} for more details about them.
\begin{cor}
  \label{c:intro:Hrushovski}
  Suppose that the homogeneous structure $M$ satisfies the assumptions of Theorem~\ref{th:intro:main}. If $M$ has the Hrushovski property, then it has the ordering property.
\end{cor}

The paper is organized as follows. In Section~\ref{sec:prel-from-model}, we recall some prerequisites from model theory, mostly about imaginaries and $M\eq$. While using standard model-theoretic terminology, we give all definitions and proofs in the language of permutation groups in the hope of making the paper more accessible to non-logicians. In Section~\ref{sec:general-de-finett}, we recall some facts from representation theory and prove Theorem~\ref{th:intro:deFinetti}. Section~\ref{sec:invar-meas-space} is devoted to the proof of Theorem~\ref{th:intro:main} and its corollaries. Finally, in Section~\ref{sec:examples}, we briefly discuss some examples and possible extensions of Theorem~\ref{th:intro:main}.

\subsection*{Acknowledgments} We would like to thank Itaï Ben Yaacov and David Evans for helping us eliminate imaginaries in some examples and Lionel Nguyen Van Thé and Andy Zucker for useful discussions. We are also grateful to the anonymous referees for a detailed reading of the paper and many helpful suggestions.

Research was partially supported by the ANR project AGRUME (ANR-17-CE40-0026) and the \emph{Investissements d'Avenir} program of Université de Lyon (ANR-16-IDEX-0005).


\section{Preliminaries from model theory}
\label{sec:prel-from-model}

We start by recalling some basic definitions. A \df{signature} $\cL$ is a collection of relation symbols $\set{R_i}$ and function symbols $\set{F_j}$, each equipped with a natural number called its \df{arity}. An \df{$\cL$-structure} is a set $M$ together with interpretations for the symbols in $\cL$: each relation symbol $R_i$ of arity $n_i$ is interpreted as an $n_i$-ary relation on $M$, that is, a subset of $M^{n_i}$, and each function symbol $F_j$ of arity $n_j$ is interpreted as a function $M^{n_j} \to M$. Functions of arity $0$ are called \df{constants}. A \df{substructure} of $M$ is a subset of $M$ closed under the functions, equipped with the induced structure. The \df{age of $M$} is the collection of isomorphism classes of all finitely generated substructures of $M$. If $\bar a$ is a tuple from $M$, we denote by $\gen{\bar a}$ the substructure of $M$ generated by $\bar a$. If the signature contains only relation symbols (which will usually be the case for us), then a substructure of $M$ is just a subset with the induced relations.

The \df{automorphism group} of $M$, $\Aut(M)$, is the group of all permutations of $M$ that preserve all relations and functions. $\Aut(M)$ is naturally a topological group if equipped with the pointwise convergence topology (where $M$ is taken to be discrete). If $M$ is countable, then $\Aut(M)$ is a Polish group. If $G = \Aut(M)$ and $A \sub M$ is a finite subset, we will denote by $G_{A}$ the pointwise stabilizer of $A$ in $G$. A basis at the identity of $G$ is given by the subgroups $\set{G_A : A \sub M \text{ is finite}}$. A topological group which admits a basis at the identity consisting of open subgroups is called \df{non-archimedean}. In particular, all groups of the form $\Aut(M)$ are non-archimedean.

The \df{type} of a tuple $\bar a \in M^k$, denoted by $\tp \bar a$, is the isomorphism type of the substructure $\gen{a_i : i < k}$ (with the $a_i$ named). Thus two tuples $\bar a$ and $\bar b$ have the same type (notation: $\bar a \equiv \bar b$) if the map $a_i \mapsto b_i$ extends to an isomorphism $\gen{\bar a} \to \gen{\bar b}$. A $k$-type is simply the type of some tuple $\bar a \in M^k$. The structure $M$ is called \df{homogeneous} if for every two tuples $\bar a$ and $\bar b$ with $\bar a \equiv \bar b$, there exists $g \in \Aut(M)$ such that $g \cdot \bar a = \bar b$. We will say that $M$ is \df{transitive} if there is only one $1$-type, i.e., $G$ acts transitively on $M$.

What we call \df{type} is usually called \df{quantifier-free type} in the model-theoretic literature. However, for homogeneous structures, which is our main interest here, the two notions coincide.

An \df{age} is a countable family of (isomorphism types of) finitely generated $\cL$-structures that is \df{hereditary} (i.e., closed under substructures) and \df{directed} (i.e., for any two structures in the class, there is another structure in the class in which they both embed). If $M$ is a given countable structure, its age is the collection of finitely generated structures that embed into it. If $M$ is homogeneous, then its age has another special property called \df{amalgamation}. An age with amalgamation is called a \df{Fraïssé class}. Fraïssé's theorem states that conversely, any Fraïssé class is the age of a unique countable, homogeneous structure, called its \df{Fraïssé limit}. Thus in order to define a homogeneous structure, one needs only to specify its age; and, as already mentioned, combinatorial properties of the age are reflected in the dynamics of the automorphism group of the limit.

The structures that will be especially important for us are the $\aleph_0$-categorical ones. A structure is \df{$\aleph_0$-categorical} if its first-order theory has a unique countable model up to isomorphism. Another characterization that will be crucial is given by the Ryll-Nardzewski theorem: $M$ is $\aleph_0$-categorical iff the diagonal action $\Aut(M) \actson M^k$ has finitely many orbits for every $k$ (a permutation group with this property is called \df{oligomorphic}). In particular, if $\cL$ is a signature that contains only finitely many relational symbols of each arity and no functions, then every homogeneous $\cL$-structure is $\aleph_0$-categorical. Conversely, if $M$ is any $\aleph_0$-categorical structure, one can render it homogeneous by expanding the signature to include all first-order formulas by a process known as Morleyization (this is another facet of the Ryll-Nardzewski theorem). As we never make assumptions about the signature, in what follows, we will tacitly assume that every $\aleph_0$-categorical structure is rendered homogeneous by this procedure. If $G$ is any closed subgroup of the full permutation group $\Sym(N)$ of some countable set $N$, one can convert $N$ into a homogeneous structure with $\Aut(N) = G$ by naming, for every $k$, each $G$-orbit on $N^k$ by a $k$-ary relation symbol. If the action $G \actson N$ is oligomorphic, then the resulting structure will be $\aleph_0$-categorical.

For the rest of the paper, we will only consider $\aleph_0$-categorical structures. In this setting, all model-theoretic information about $M$ is captured by the actions $\Aut(M) \actson M^k$. We refer the reader to Hodges~\cite{Hodges1993} for more details on Fraïssé theory, $\aleph_0$-categorical structures, and their automorphism groups.

Let $M$ be $\aleph_0$-categorical, $G = \Aut(M)$, and let $A \sub M$ be finite. The \df{algebraic closure of $A$} (denoted $\acl(A)$) is the union of all finite orbits of $G_A$ on $M$. We will say that $M$ \df{has no algebraicity} if the algebraic closure is trivial, that is, $\acl(A) = A$ for all finite $A \sub M$. By Neumann's lemma \cite{Hodges1993}*{Lemma~4.2.1}, having no algebraicity is equivalent to the following: for all finite $A, B, C \sub M$ with $A \cap B \cap C = \emptyset$, there exists $g \in G_C$ such that $g \cdot A \cap B = \emptyset$.

An \df{imaginary element} of $M$ is the equivalence class of a tuple $\bar a \in M^k$ for some $G$-invariant equivalence relation on $M^k$. We denote by $M\eq$ the collection of all imaginaries. In symbols,
\begin{equation*}
  M\eq = \bigsqcup \set{M^k / E : k \in \N \And E \text{ is a $G$-invariant equivalence relation on } M^k}.
\end{equation*}
It is clear that $G$ also acts on $M\eq$ and, moreover, the action $G \actson M\eq$ is \df{locally oligomorphic}, i.e., it is oligomorphic on any union of finitely many $G$-orbits (see, e.g., \cite{Tsankov2012}*{Theorem~2.4}). Open subgroups of $G$ are precisely the stabilizers of imaginary elements of $M$. On the one hand, if $e = [\bar a]_E \in M\eq$ for some $\bar a \in M^k$, then $G_{\bar a} \leq G_e$, which implies that $G_e$ is open. On the other, if $V \leq G$ is an open subgroup, there exists a tuple $\bar a \in M^k$ such that $G_{\bar a} \leq V$. Then a $G$-invariant equivalence relation $E$ on $G \cdot \bar a$ can be defined by
\begin{equation}
  \label{eq:def-imaginary-eqrel}
  (g_1 \cdot \bar a) \eqrel{E} (g_2 \cdot \bar a) \iff g_1V = g_2V \quad \text{ for } g_1, g_2 \in G
\end{equation}
and it can be extended by equality to the rest of $M^k$. If we set $e = [\bar a]_E$, we have that $G_e = V$. This gives another possible way to view the set of imaginary elements of $M$ as
\begin{equation*}
  \bigsqcup \set{G/V : V \text{ is an open subgroup of } G}.
\end{equation*}
Note, however, that there is no canonical bijection between this set and $M\eq$ according to our definition even though they are interdefinable.

We can define for a finite $A \sub M\eq$,
\begin{equation*}
  \acl\eq A = \set{e \in M\eq : G_A \cdot e \text{ is finite}}.
\end{equation*}
Similarly, we can define the \df{definable closure} as
\begin{equation*}
  \dcl\eq A = \set{e \in M\eq : G_A \cdot e = \set{e}}.
\end{equation*}
For arbitrary $A \sub M\eq$, we define $\acl\eq A$ to be the union of $\acl\eq A'$ over all finite $A' \sub A$. Similarly for $\dcl\eq$. A subset $A \sub M\eq$ is \df{algebraically closed} if $\acl\eq A = A$. In other words, $A$ is algebraically closed if for all finite $A' \sub A$, $G_{A'}$ has only infinite orbits outside of $A$.
We have the following basic properties of the algebraic closure.
\begin{lemma}
  \label{l:acl}
  The following hold for an $\aleph_0$-categorical $M$:
  \begin{enumerate}
  \item \label{i:acl:1} For all $A \sub M\eq$, $\acl\eq A$ is algebraically closed;
  \item \label{i:acl:2} If $A, B \sub M\eq$ are algebraically closed, then so is $A \cap B$.
  \end{enumerate}
\end{lemma}
\begin{proof}
  \ref{i:acl:1} A permutation group theoretic proof of this fact can be found for example in \cite{Evans2016}*{Lemma~2.4}.

  \ref{i:acl:2} Suppose that $C \sub A \cap B$ is finite and $e \in M\eq$ is such that $G_C \cdot e$ is finite. Then, as $C \sub A$ and $A$ is algebraically closed, we have that $e \in A$, and similarly, $e \in B$.
\end{proof}

$M$ admits \df{elimination of imaginaries} if all imaginary elements are \df{interdefinable} with real tuples, that is, for every $e \in M\eq$, there exists $k \in \N$ and a tuple $\bar a \in M^k$ such that $e \in \dcl\eq \bar a$ and $\bar a \in \dcl\eq e$, or equivalently, $G_e = G_{\bar a}$. $M$ admits \df{weak elimination of imaginaries} if for every imaginary element $e \in M\eq$, there exists a real tuple $\bar a \in M^k$ such that $e \in \dcl\eq \bar a$ and $\bar a \in \acl\eq e$. Equivalently, for every open subgroup $V \leq G$, there exists $k$ and a tuple $\bar a \in M^k$ such that $G_{\bar a} \leq V$ and $[V : G_{\bar a}] < \infty$. In these definitions, it is important to allow $k = 0$ and $\bar a = \emptyset$.

The two hypotheses of no algebraicity and weak elimination of imaginaries combined give us a complete understanding of the $\acl\eq$ operator.
\begin{lemma}
  \label{l:no-algeb-and-welim}
  Suppose that $M$ is $\aleph_0$-categorical and that it has no algebraicity and admits weak elimination of imaginaries. Then for all $A, B \sub M$, we have that
  \begin{equation*}
   \acl\eq A \cap \acl\eq B = \dcl\eq (A \cap B). 
  \end{equation*}
\end{lemma}
\begin{proof}
  The $\supseteq$ inclusion being clear, we only check the other. We may assume that $A$ and $B$ are finite. Suppose that $e \in \acl\eq A$. By weak elimination of imaginaries, there exists a tuple $\bar c$ such that $e \in \dcl\eq \bar c$ and $\bar c \in \acl\eq e$. We will show that the tuple $\bar c$ is contained in $A$. Consider the group $H = G_{A \cup \set{e}}$. As $H$ contains the open subgroup $G_{A \bar c}$, it is also open. By weak elimination of imaginaries, there exists a tuple $\bar a$ such that $G_{\bar a} \leq H$ and $[H : G_{\bar a}] < \infty$. As $e \in \acl\eq A$, $[G_A : H] < \infty$ and thus $[G_A : G_{\bar a}] < \infty$. By the no algebraicity assumption, $\bar a$ must be contained in $A$, so, in particular, $H = G_A$, i.e., $G_A$ fixes $e$. If $\bar c$ is not contained in $A$, then the orbit $G_A \cdot \bar c$ is infinite, which implies that the orbit $G_e \cdot \bar c$ is infinite, contradicting the fact that $\bar c \in \acl\eq e$. Thus we conclude that $\bar c$ is contained in $A$. An analogous argument shows that $\bar c$ is also contained in $B$ and hence, $e \in \dcl\eq (A \cap B)$.
\end{proof}


\section{Unitary representations and a generalization of de Finetti's theorem}
\label{sec:general-de-finett}

Recall that a \df{unitary representation} of a topological group $G$ is a continuous action on a complex Hilbert space $\cH$ by unitary operators, or, equivalently, a continuous homomorphism from $G$ to the unitary group of $\cH$. A representation $G \actson \cH$ is \df{irreducible} if $\cH$ contains no non-trivial, $G$-invariant, closed subspaces.

In the case where $M$ is an $\aleph_0$-categorical structure and $G = \Aut(M)$, the action $G \actson M\eq$ gives rise to a representation $G \actson^\lambda \ell^2(M\eq)$ given by
\begin{equation*}
  (\lambda(g) \cdot f)(e) = f(g^{-1} e), \quad \text{ where } f \in \ell^2(M\eq), g \in G, e \in M\eq.
\end{equation*}
It turns out that this representation captures all of the representation theory of $G$. More precisely, it follows from the results of \cite{Tsankov2012} that the following holds.
\begin{fact}
  \label{f:representations}
  Let $M$ be an $\aleph_0$-categorical structure and let $G = \Aut(M)$. Then every unitary representation of $G$ is a sum of irreducible representations and every irreducible representation is isomorphic to a subrepresentation of $\lambda$. In particular, every representation of $G$ is a subrepresentation of a direct sum of copies of $\lambda$.
\end{fact}
\begin{proof}
  The first claim is part of the statement of \cite{Tsankov2012}*{Theorem~4.2}. For the second, it follows from \cite{Tsankov2012}*{Theorem~4.2} that every irreducible representation of $G$ is an induced representation of the form $\Ind_H^G(\sigma)$, where $H$ is an open subgroup of $G$ and $\sigma$ is an irreducible representation of $H$ that factors through a finite quotient $K = H/V$ of $H$. (We refer the reader to \cite{Tsankov2012} for the definition of induced representation and more details.)  As $V \leq G$ is open, there exists a tuple $\bar a$ from $M$ such that $G_{\bar a} \leq V$. As in \eqref{eq:def-imaginary-eqrel}, define the $G$-invariant equivalence relation $E$ on $G \cdot \bar a$ by
  \begin{equation*}
    (g_1 \cdot \bar a) \eqrel{E} (g_2 \cdot \bar a) \iff g_1V = g_2V \quad \text{ for } g_1, g_2 \in G
  \end{equation*}
  and note that $V = G_{[\bar a]_E}$. In particular, the quasi-regular representation $\ell^2(G/V)$ is isomorphic to the subrepresentation $\ell^2(G \cdot [\bar a]_E)$ of $\ell^2(M\eq)$. On the other hand,
  \begin{equation*}
    \ell^2(G/V) \cong \Ind_V^G(1_V) \cong \Ind_H^G(\Ind_V^H(1_V)) \cong \Ind_H^G(\lambda_K),
  \end{equation*}
  where $\lambda_K$ denotes the left-regular representation of $K$. As $\sigma$ (being an irreducible representation of the finite group $K$) is a subrepresentation of $\lambda_K$, the result follows.
\end{proof}

If $\cH$ is a Hilbert space and $\cH_1, \cH_2, \cH_3$ are subspaces with $\cH_2 \sub \cH_1 \cap \cH_3$, we write $\cH_1 \Perp[\cH_2] \cH_3$ if $\cH_1 \ominus \cH_2 \perp \cH_3 \ominus \cH_2$, where $\cH_i \ominus \cH_2$ denotes the orthogonal complement of $\cH_2$ in $\cH_i$.
If we let $p_1, p_2, p_3$ denote the corresponding orthogonal projections, this is equivalent to $p_3 p_1 = p_2 p_1$. Note that $\cH_1 \Perp[\cH_2] \cH_3$ implies that $\cH_2 = \cH_1 \cap \cH_3$.

If $G \actson \cH$ is a unitary representation of $G$ and $A \sub M\eq$, let
\begin{equation}
  \label{eq:defHA}
  \cH_A = \cl{\set{\xi \in \cH : G_{A'} \cdot \xi = \xi \text{ for some finite } A' \sub A}}.
\end{equation}
It is clear that $\cH_A$ is a closed subspace of $\cH$.

\begin{prop}
  \label{p:repr-independent}
  Let $M$ be $\aleph_0$-categorical and $G = \Aut(M)$. Let $A$ and $B$ be algebraically closed subsets of $M\eq$. Then $\cH_{A} \Perp[\cH_{A \cap B}] \cH_{B}$.
\end{prop}
\begin{proof}
  As for any subset $C \sub M\eq$, the projection $p_C$ onto $\cH_C$ commutes with 
  direct sums and subrepresentations, by Fact~\ref{f:representations}, we can reduce to the case where $\cH = \ell^2(M\eq)$ and $\pi = \lambda$. If $\xi \in \cH$, we view it as a function $M\eq \to \C$ and we let $\supp \xi = \set{e \in M\eq : \xi(e) \neq 0}$.

  The main observation is the following: if $C \sub M\eq$ is algebraically closed, then
  \begin{equation*}
    \cH_C = \set{\xi \in \cH : \supp \xi \sub C }.
  \end{equation*}
  The $\supseteq$ inclusion follows from the fact that vectors with finite support are dense. For the other inclusion, as the subspace on the right-hand side is closed, it suffices to see that for all finite $C' \sub C$ and all $\xi$ fixed by $G_{C'}$, $\supp \xi \sub C$. Let $e \in M\eq$ be such that $\xi(e) \neq 0$. As $\xi$ is fixed by $G_{C'}$, it must be constant on the orbit $G_{C'} \cdot e$. As $\xi$ is in $\ell^2$, this implies that $G_{C'} \cdot e$ is finite, i.e., $e \in \acl\eq C' \sub C$.

  Now it follows from the hypothesis and Lemma~\ref{l:acl} that
  \begin{align*}
    \cH_A \ominus \cH_{A \cap B} &= \set{\xi \in \cH : \supp \xi \sub A \sminus B} \quad \And \\
    \cH_B \ominus \cH_{A \cap B} &= \set{\xi \in \cH : \supp \xi \sub B \sminus A},
  \end{align*}
  whence the result.
\end{proof}
\begin{remark}
  \label{rem:mt-approach}
  A more model-theoretic treatment of similar ideas, using the formalism of semigroups of projections, can be found in \cite{BenYaacov2018}.
\end{remark}

Now consider a measure-preserving action $G \actson (X, \mu)$, where $(X, \mu)$ is a probability space. As $G$ is not locally compact, one has to take some care how this is defined. We denote by $\MALG(X, \mu)$ the Boolean algebra of all measurable subsets of $X$ with two such sets identified if their symmetric difference has measure $0$. $\MALG(X, \mu)$ is naturally a metric space with the distance between $A$ and $B$ given by $\mu(A \sdiff B)$. We denote by $\Aut(X, \mu)$ the group of all isometric automorphisms of $\MALG(X, \mu)$, that is, the group of all automorphisms of the Boolean algebra that also preserve the measure. $\Aut(X, \mu)$ is naturally a topological group if equipped with the pointwise convergence topology coming from its action on $\MALG(X, \mu)$. If $(X, \mu)$ is \df{standard} (i.e., $X$ is a standard Borel space and $\mu$ is a Borel probability measure), then $\Aut(X, \mu)$ is a Polish group. For us, a \df{measure-preserving action $G \actson (X, \mu)$} will mean a continuous homomorphism $G \to \Aut(X, \mu)$, that is, $G$ acts on measurable sets and measurable functions (up to measure $0$) but not necessarily on points. It is easy to see that if $X$ is standard and one has a jointly measurable action on points $G \actson X$ that preserves the measure $\mu$, then this gives an action in our sense. The converse is also true for non-archimedean groups but this is less obvious (see \cite{Glasner2005a}*{Theorem~2.3}) and we will not need it.

If $\cF_1, \cF_2, \cG$ are $\sigma$-fields in a probability space, we will denote by $\cF_1 \indep[\cG] \cF_2$ the fact that $\cF_1$ and $\cF_2$ are \df{conditionally independent over $\cG$}, i.e., $\E(\xi \mid \cF_2\cG) = \E(\xi \mid \cG)$ for every $\cF_1$-measurable random variable $\xi$. If $\cG$ is trivial, we will write simply $\cF_1 \indep \cF_2$ and will say that $\cF_1$ and $\cF_2$ are independent. We will often make use of the standard facts about conditional independence that go in model theory by the name of \emph{forking calculus}. A general source is \cite{Kallenberg2002}*{Chapter~5} and we will give references to precise statements where needed.

If $G = \Aut(M)$ and a measure-preserving action $G \actson (X, \mu)$ is given, for $A \sub M\eq$, we denote by $\cF_A$ the $\sigma$-field of measurable subsets of $X$ generated by the $G_{A'}$-fixed subsets for all finite $A' \sub A$. The following is the main result of this section. 
\begin{theorem}
  \label{th:general-independence}
  Let $M$ be an $\aleph_0$-categorical structure and let $G = \Aut(M)$. Let $G \actson (X, \mu)$ be any measure-preserving action on a probability space. Then the following hold:
  \begin{enumerate}
  \item \label{i:thgen:1}   For all algebraically closed $A, B \sub M\eq$, we have that $\cF_A \indep[\cF_{A \cap B}] \cF_B$.

  \item \label{i:thgen:2} If $M$ has no algebraicity and admits weak elimination of imaginaries, then for all $A, B \sub M$, we have that $\cF_A \indep[\cF_{A \cap B}] \cF_B$.
  \end{enumerate}
\end{theorem}
\begin{proof}
  \ref{i:thgen:1}
  Consider the \df{Koopman representation} $G \actson^\pi L^2(X)$ given by
  \begin{equation*}
    (\pi(g) \cdot f)(x) = f(g^{-1} \cdot x), \quad \text{ where } f \in L^2(X), g \in G, x \in X.
  \end{equation*}
  For $C \sub M\eq$, we denote by $L^2(\cF_C)$ the subspace of $L^2(X)$ consisting of all $\cF_C$-measurable functions. Observe that if we write $\cH = L^2(X)$, then $L^2(\cF_C) = \cH_C$ (as defined in \eqref{eq:defHA}). To see the $\sub$ inclusion, note that functions of the form $F(\eta_0, \ldots, \eta_{n-1})$, where $F \colon \R^n \to \C$ is a bounded measurable function and each $\eta_i \colon X \to \R$ is a random variable fixed by $G_{C_i}$ for some finite $C_i \sub C$, are dense in $L^2(\cF_C)$;  each such function is in $L^2$ and is fixed by $G_{\bigcup_i C_i}$, so belongs to $\cH_C$. The inverse inclusion follows from the fact that if $\xi \in \cH_C$ is fixed by $G_{C'}$ for some finite $C' \sub C$, then $\xi$ is measurable with respect to the $\sigma$-field generated by the events $\set[\big]{\set{\xi > r} : r \in \R}$ and all of these belong to $\cF_C$.

  To show the required independence, it suffices to see that for all $\eta_A \in L^2(\cF_A)$, we have that
  \begin{equation*}
    \E(\eta_A \mid \cF_B) = \E(\eta_A \mid \cF_{A \cap B})
  \end{equation*}
  (see, e.g., \cite{Kallenberg2002}*{Proposition~5.6}). Recalling that the conditional expectation $\E(\cdot \mid \cF_C)$ for functions in $L^2$ is just the projection operator onto $\cH_C$, this follows directly from Proposition~\ref{p:repr-independent}.

  \ref{i:thgen:2} Denote $A' = \acl\eq A$, $B' = \acl\eq B$, $C = \dcl\eq (A \cap B)$. By Lemma~\ref{l:acl}, $A'$ and $B'$ are algebraically closed and by Lemma~\ref{l:no-algeb-and-welim}, we have that $C = A' \cap B'$. Now \ref{i:thgen:1} applied to $A'$ and $B'$ yields that $\cF_{A'} \indep[\cF_C] \cF_{B'}$. It only remains to observe that $\cF_C = \cF_{A \cap B}$ and that $\cF_A \sub \cF_{A'}$, $\cF_B \sub \cF_{B'}$.
\end{proof}

Theorem~\ref{th:general-independence}  has the following immediate corollary, which can be viewed as a generalization of the classical theorem of de Finetti.
\begin{cor}
  \label{c:deFinetti}
  Let $M$ be an $\aleph_0$-categorical structure with no algebraicity that admits weak elimination of imaginaries and let $G = \Aut(M)$. Consider a family of random variables $\set{\xi_a : a \in M}$ whose joint distribution is invariant under $G$. Then these variables are conditionally independent over the $G$-invariant $\sigma$-field. If the $G$-invariant $\sigma$-field is trivial and $M$ is transitive, then the $\xi_a$ are \df{i.i.d.} (independent, identically distributed).
\end{cor}
\begin{proof}
  Let $a_1, \ldots, a_n, b_1, \ldots, b_m \in M$ with $\set{a_1, \ldots, a_n} \cap \set{b_1, \ldots, b_m} = \emptyset$. Then it follows from Theorem~\ref{th:general-independence} \ref{i:thgen:2} that
  \begin{equation*}
    (\xi_{a_1},\ldots, \xi_{a_n}) \indep[\mathcal{F}_\emptyset] (\xi_{b_1},\ldots, \xi_{b_m})
  \end{equation*}
  and it remains to observe that $\cF_\emptyset$ is precisely the $G$-invariant $\sigma$-field.

  If $M$ is moreover transitive, the variables $\xi_a$ must have the same distribution by $G$-invariance.
\end{proof}

\begin{remark}
  \label{rem:dF-weaker-assumption}
  Using the properties of independence and an inductive argument, it is possible to replace the no algebraicity and weak elimination of imaginaries assumption above with a slightly weaker one. Namely, we only need that $\acl\eq A \cap \acl\eq B = \dcl\eq \emptyset$ for any disjoint $A$ and $B$ where $B$ is a singleton, rather than for arbitrary $A$ and $B$.
\end{remark}

An action $G \actson (X, \mu)$ is called \df{ergodic} if the $G$-invariant $\sigma$-field is trivial. Thus in the case of ergodic actions, one obtains genuine independence in Corollary~\ref{c:deFinetti}.

Another interesting remark is that by virtue of Fraïssé's theorem, Corollary~\ref{c:deFinetti} can be applied even in situations in which there is no obvious group around. This is best illustrated by the following example, which is a well-known theorem of Ryll-Nardzewski~\cite{RyllNardzewski1957}. If $\bar \xi = (\xi_0, \ldots, \xi_{n-1})$ and $\bar \eta = (\eta_0, \ldots, \eta_{n-1})$ are tuples of random variables, we use the notation $\bar \xi \equiv \bar \eta$ to signify that $\bar \xi$ and $\bar \eta$ have the same distribution.
\begin{cor}[Ryll-Nardzewski]
  \label{c:RyllNardzewski}
  Let $\mu$ be a Borel probability measure on $\R^\N$ and denote by $\xi_i \colon \R^\N \to \R$ the projection on the $i$-th coordinate. Suppose that for all $i_0 < \cdots < i_{k-1}$, we have that $(\xi_{i_0}, \ldots, \xi_{i_{k-1}}) \equiv (\xi_0, \ldots, \xi_{k-1})$. Denote by $\phi \colon \R^\N \to \R^\N$ the one-sided shift defined by $\phi(x_0, x_1, \ldots) = (x_1, x_2, \ldots)$ and suppose moreover that $\mu$ is $\phi$-ergodic. Then $\mu$ is a product measure.
\end{cor}
\begin{proof}
  Here the relevant structure is $(\N, <)$ which has no automorphisms. Its age is the class of finite linear orders. This age amalgamates and its Fraïssé limit is the countable dense linear order without endpoints $(\Q, <)$, which satisfies the hypothesis of Corollary~\ref{c:deFinetti}. Consider the random variables $(\xi_a : a \in \Q)$ whose distribution is defined by
  \begin{equation*}
    (\xi_{a_0}, \ldots, \xi_{a_{k-1}}) \equiv (\xi_0, \ldots, \xi_{k-1}) \quad \text{ for all } a_0 < \cdots < a_{k-1} \in \Q.
  \end{equation*}
  In order to apply Corollary~\ref{c:deFinetti} and conclude, we only need to check that the $\Aut(\Q)$-invariant $\sigma$-field is trivial. Let $\cF_n$ be the $\sigma$-field generated by $\xi_0, \ldots, \xi_{n-1}$. Suppose that $S$ is an $\Aut(\Q)$-invariant event and fix $\eps > 0$. Then there exists $n \in \N$, $a_0 <  \cdots < a_{n-1} \in \Q$ and an event $S_\eps$ measurable with respect to the $\sigma$-field generated by $\xi_{a_0}, \ldots, \xi_{a_{n-1}}$ such that $\bP(S \sdiff S_\eps) < \eps$. Let $g \in \Aut(\Q)$ be such that $g \cdot (a_0, \ldots, a_{n-1}) = (0, \ldots, n-1)$. Then $g \cdot S_\eps$ is $\cF_n$-measurable and by the invariance of $S$, we have that $\bP \big(S \sdiff (g \cdot S_\eps) \big) < \eps$. As $\eps$ was arbitrary, we conclude that $S$ is measurable with respect to the original $\sigma$-field $\bigvee_n \cF_n$. As $\phi$ extends to an automorphism of $\Q$, we obtain that $\phi^{-1}(S) = S$. Now the fact that $\mu$ is $\phi$-ergodic allows us to conclude that $S$ or its complement is $\mu$-null.
\end{proof}


\section{Invariant measures on the space of linear orderings}
\label{sec:invar-meas-space}

In this section, we fix a homogeneous $\aleph_0$-categorical structure $M$ with no algebraicity that admits weak elimination of imaginaries and we let $G = \Aut(M)$. We denote by $\LO(M)$ the space of all linear orders on $M$, that is
\begin{equation*}
  \LO(M) = \set{x \in 2^{M \times M} : x \text{ is a linear order}}.
\end{equation*}
$\LO(M)$ is a closed subset of $2^{M \times M}$ and thus a compact space. If $x \in \LO(M)$, we will use the more traditional infix notation $a <_x b$ instead of $(a, b) \in x$. $\Sym(M)$ (and, in particular, $G$) acts naturally on $\LO(M)$ as follows:
\begin{equation*}
  a <_{g \cdot x} b \iff g^{-1} \cdot a <_x g^{-1} \cdot b.
\end{equation*}
Our goal is to study the $G$-invariant measures on $\LO(M)$. There is always at least one such measure $\muu$ which is, in fact, invariant under all of $\Sym(M)$. It is defined by
\begin{equation*}
  \muu(a_0 <_x \cdots <_x a_{k-1}) = 1/k! \quad \text{ for all distinct } a_0, \ldots, a_{k-1} \in M.
\end{equation*}
Here and below, we employ the usual notation from probability theory and write $a_0 <_x \cdots <_x a_{k-1}$ for the event $\set{x \in \LO(M) : a_0 <_x \cdots <_x a_{k-1}}$. We will call $\muu$ the \df{uniform measure}. Glasner and Weiss~\cite{Glasner2002a} have shown that it is the only measure invariant under the whole symmetric group. The proof is simple: the way the tuple $(a_0, \ldots, a_{k-1})$ is ordered gives a partition of $\LO(M)$ into $k!$ pieces and for every two elements of this partition, there is an element of $\Sym(M)$ that sends one to the other, so they must all have the same measure.

Our main theorem is the following (cf. Theorem~\ref{th:intro:main} in the introduction):
\begin{theorem}
  \label{Thm:Princip} Let $M$ be a transitive, $\aleph_0$-categorical structure with no algebraicity that admits weak elimination of imaginaries and let $G=\Aut(M)$. Consider the action $G \actson \LO(M)$. Then exactly one of the following holds:
  \begin{enumerate}
  \item The action $G \actson \LO(M)$ has a fixed point (i.e., there is a definable linear order on $M$);
  \item $\muu$ is the unique $G$-invariant measure on $\LO(M)$.
  \end{enumerate}
\end{theorem}

We describe a method to construct the uniform measure on $\LO(M)$ that will help illustrate our strategy for the proof. Let
\begin{equation*}
  \Omega = \set{z \in [0, 1]^M : z(a) \neq z(b) \text{ for all } a \neq b}
\end{equation*}
and define the map $\pi \colon \Omega \to \LO(M)$ by
\begin{equation}
  \label{eq:def-pi}
  a <_{\pi(z)} b \iff z(a) < z(b) \quad \text{ for } a, b \in M.
\end{equation}
The group $G$ acts naturally on $[0, 1]^M$, $\Omega$ is a $G$-invariant set, and the map $\pi$ is $G$-equivariant. Thus any $G$-invariant measure on $\Omega$ gives rise, via $\pi$, to a $G$-invariant measure on $\LO(M)$. In view of Corollary~\ref{c:deFinetti}, the only $G$-invariant, ergodic measures on $[0, 1]^M$ are of the form $\lambda^M$, where $\lambda$ is a measure on $[0, 1]$. It is clear that $(\lambda^M)(\Omega) = 1$ iff $\lambda$ is non-atomic and in that case, $\pi_*(\lambda^M) = \muu$ (this is true because $\lambda^M$ is $\Sym(M)$-invariant and as we noted above, $\muu$ is the only $\Sym(M)$-invariant measure on $\LO(M)$). What we aim to show below is that if $M$ does not admit a $G$-invariant linear order, then the map $\pi$ is invertible almost everywhere for \emph{any} $G$-invariant, ergodic measure on $\LO(M)$.

Let $\mu$ be an ergodic, $G$-invariant measure on $\LO(M)$. We will use probabilistic notation: we will denote by $<_x$ (or only by $<$ if there is no danger of confusion) a random element of $\LO(M)$ chosen according to $\mu$, by $\bP$ the probability of events and by $\E$ the expectation. If $A$ is an event, we denote by $\bOne_A$ its characteristic function. For every $a \in M$, we denote by $\cF_a$ the $\sigma$-field fixed by $G_a$. 

For every $2$-type $\tau$ and every $a \in M$, we define a random variable $\eta^\tau_a$ by
\begin{equation*}
  \eta^\tau_a = \Pr(c < a \mid \cF_a), \quad \text{ where } \tp(ac) = \tau.
\end{equation*}
The definition above does not depend on $c$ but only on $\tau$. Indeed, if $c' \in M$ is another element with $\tp ac' = \tau$, and $\zeta_a^{\tau} = \Pr(c' < a \mid \cF_a)$, then for every $\xi \in L^2(\cF_a)$, invariance implies that $\ip{\xi, \bOne_{c < a}} = \ip{\xi, \bOne_{c' < a}}$, so $\eta_a^\tau = \zeta_a^\tau$ a.s.

\begin{lemma}
  \label{Lem:i.i.d.}
  The random variables $(\eta^\tau_a)_{a\in M}$ are i.i.d.
\end{lemma}
\begin{proof}
  This is a direct consequence of Corollary~\ref{c:deFinetti}.
\end{proof}

The following is a basic fact about conditional expectation that we will need.
\begin{lemma}
  \label{l:strictly-positive}
  Let $X \geq 0$ be an integrable random variable, $A$ be an event and $\cF$ be a $\sigma$-field. Suppose that $X > 0$ on $A$. Then $\E(X \mid \cF) > 0$ on $A$ a.s.
\end{lemma}
\begin{proof}
  Let $Y = \E(X \mid \cF)$. If the conclusion of the lemma is false, there exists an event $C \sub A$ such that $\Pr(C) > 0$ and $\int_C Y = 0$. In particular, $C \sub \set{Y = 0}$. As the set $\set{Y = 0}$ is in $\cF$, we have:
  \begin{equation*}
    0 < \int_C X \leq \int_{Y = 0} X = \int_{Y = 0} Y = 0,
  \end{equation*}
  contradiction.
\end{proof}

If $\tau$ is a $2$-type and $a \in M$, we define
\begin{equation*}
  D_\tau(a) = \set{b \in M : \tp ab = \tau}.
\end{equation*}
The next lemma is the main tool that allows us to recover the order from the random variables $(\eta^\tau_a)_{a \in M}$.
\begin{lemma}
  \label{Lem:Inters}
  Let the type $\tau$ and $a, b \in M$ be such that $D_\tau(a) \cap D_\tau(b) \neq \emptyset$. Then almost surely,
  \begin{equation*}
    a<b \implies \eta_a^\tau \leq \eta_b^\tau.
  \end{equation*}
  Moreover, for any $c \in D_\tau(a) \cap D_\tau(b)$, we have that almost surely,
  \begin{equation*}
  a < c < b \implies \eta^\tau_a < \eta^\tau_b.
  \end{equation*}
\end{lemma}
\begin{proof}
It follows from Theorem~\ref{th:general-independence} that for distinct $a, b, c \in M$,
\begin{equation*}
a < b, \cF_b \indep[\cF_a] c < a,
\end{equation*}
so, using the chain rule \cite{Kallenberg2002}*{Proposition~5.8}, we obtain that
\begin{equation}
  \label{eq:cond-indep}
  a < b \indep[\cF_a \cF_b] c < a.
\end{equation}
  
  Let $c \in D_\tau(a) \cap D_\tau(b)$. Using the fact that $\cF_a \indep[\cF_b] c < b$ (which follows from Theorem~\ref{th:general-independence} and the chain rule), \eqref{eq:cond-indep}, and their variants obtained by exchanging $a$ and $b$, we have:
  \begin{equation*}
    \begin{split}
      \E(\bOne_{a < b} \mid \cF_a \cF_b) ( \eta^\tau_b - \eta^\tau_a ) &= \E(\bOne_{a < b} \mid \cF_a \cF_b) \big(\Pr(c < b \mid \cF_b) - \Pr(c < a \mid \cF_a) \big) \\
      &= \E(\bOne_{a < b} \mid \cF_a \cF_b) \big(\E(\bOne_{c < b} \mid \cF_b \cF_a) - \E(\bOne_{c < a} \mid \cF_a \cF_b) \big) \\
      &= \E(\bOne_{a < b}(\bOne_{c < b} - \bOne_{c < a}) \mid \cF_a \cF_b ).
    \end{split}
  \end{equation*}
  By Lemma~\ref{l:strictly-positive}, $\E(\bOne_{a < b} \mid \cF_a \cF_b)$ is a.s. strictly positive on the event $a < b$. This implies that on $a < b$, we have that
  \begin{equation*}
    \eta^\tau_b - \eta^\tau_a = \frac{\E(\bOne_{a < b}(\bOne_{c < b} - \bOne_{c < a}) \mid \cF_a \cF_b )}{\E(\bOne_{a < b} \mid \cF_a \cF_b)}.
  \end{equation*}
  The numerator is always non-negative, so $\eta^\tau_a \leq \eta^\tau_b$ on $a < b$. For the second assertion of the lemma, note that on $a < c < b$, we have that $\bOne_{a < b}(\bOne_{c < b} - \bOne_{c < a}) = 1$. Thus Lemma~\ref{l:strictly-positive} applies again and the numerator is also strictly positive on that event.
\end{proof}

We will also need a combinatorial fact about $2$-types. For a $2$-type $\tau$ and $a, b \in M$, we say that $y_0, y_1, \ldots, y_{2n}$ is an  \df{alternating $\tau$-path} (or just a \df{$\tau$-path} for brevity) between $a$ and $b$ if $y_0=a$, $y_{2n}=b$ and $\tp(y_{2i} y_{2i+1}) = \tp(y_{2i+2} y_{2i+1}) = \tau$ for all $i = 0, \ldots, n-1$ and all of the nodes of the path are distinct (see Figure~\ref{fig:tau-path}). The \df{interior} of the path is the collection of all nodes except its endpoints. Being a $\tau$-path is, of course, a $G$-invariant condition.

\begin{figure}
\tikzset{every picture/.style={line width=0.75pt}} 

\begin{tikzpicture}[x=0.75pt,y=0.75pt,yscale=-1,xscale=1]

\draw    (280.36,72.93) -- (319.18,20.89) ;
\draw [shift={(320.38,19.28)}, rotate = 486.72] [color={rgb, 255:red, 0; green, 0; blue, 0 }  ][line width=0.75]    (10.93,-3.29) .. controls (6.95,-1.4) and (3.31,-0.3) .. (0,0) .. controls (3.31,0.3) and (6.95,1.4) .. (10.93,3.29)   ;
\draw    (260.35,72.93) -- (221.52,20.89) ;
\draw [shift={(220.33,19.28)}, rotate = 413.28] [color={rgb, 255:red, 0; green, 0; blue, 0 }  ][line width=0.75]    (10.93,-3.29) .. controls (6.95,-1.4) and (3.31,-0.3) .. (0,0) .. controls (3.31,0.3) and (6.95,1.4) .. (10.93,3.29)   ;
\draw    (160.3,72.93) -- (199.12,20.89) ;
\draw [shift={(200.32,19.28)}, rotate = 486.72] [color={rgb, 255:red, 0; green, 0; blue, 0 }  ][line width=0.75]    (10.93,-3.29) .. controls (6.95,-1.4) and (3.31,-0.3) .. (0,0) .. controls (3.31,0.3) and (6.95,1.4) .. (10.93,3.29)   ;
\draw    (440.44,72.93) -- (479.27,20.89) ;
\draw [shift={(480.46,19.28)}, rotate = 486.72] [color={rgb, 255:red, 0; green, 0; blue, 0 }  ][line width=0.75]    (10.93,-3.29) .. controls (6.95,-1.4) and (3.31,-0.3) .. (0,0) .. controls (3.31,0.3) and (6.95,1.4) .. (10.93,3.29)   ;
\draw    (544,77) -- (502.63,22.03) ;
\draw [shift={(501.43,20.43)}, rotate = 413.03999999999996] [color={rgb, 255:red, 0; green, 0; blue, 0 }  ][line width=0.75]    (10.93,-3.29) .. controls (6.95,-1.4) and (3.31,-0.3) .. (0,0) .. controls (3.31,0.3) and (6.95,1.4) .. (10.93,3.29)   ;
\draw    (380.41,72.93) -- (341.59,20.89) ;
\draw [shift={(340.39,19.28)}, rotate = 413.28] [color={rgb, 255:red, 0; green, 0; blue, 0 }  ][line width=0.75]    (10.93,-3.29) .. controls (6.95,-1.4) and (3.31,-0.3) .. (0,0) .. controls (3.31,0.3) and (6.95,1.4) .. (10.93,3.29)   ;

\draw (400.21,35.51) node [anchor=north west][inner sep=0.75pt]  [rotate=-0.57,xscale=0.75,yscale=0.75] [align=left] {$\displaystyle \cdots $};
\draw (146,74.06) node [anchor=north west][inner sep=0.75pt]  [xscale=0.75,yscale=0.75]  {$a$};
\draw (204.03,5.1) node [anchor=north west][inner sep=0.75pt]  [xscale=0.75,yscale=0.75]  {$y_{1}$};
\draw (162.01,36.5) node [anchor=north west][inner sep=0.75pt]  [xscale=0.75,yscale=0.75]  {$\tau $};
\draw (262.36,75.06) node [anchor=north west][inner sep=0.75pt]  [xscale=0.75,yscale=0.75]  {$y_{2}$};
\draw (323.1,2.74) node [anchor=north west][inner sep=0.75pt]  [xscale=0.75,yscale=0.75]  {$y_{3}$};
\draw (382.42,75.06) node [anchor=north west][inner sep=0.75pt]  [xscale=0.75,yscale=0.75]  {$y_{4}$};
\draw (428.15,73.06) node [anchor=north west][inner sep=0.75pt]  [xscale=0.75,yscale=0.75]  {$y_{2n-2}$};
\draw (483.18,4.21) node [anchor=north west][inner sep=0.75pt]  [xscale=0.75,yscale=0.75]  {$y_{2n-1}$};
\draw (542.5,75.17) node [anchor=north west][inner sep=0.75pt]  [xscale=0.75,yscale=0.75]  {$b$};
\draw (365.12,36.5) node [anchor=north west][inner sep=0.75pt]  [xscale=0.75,yscale=0.75]  {$\tau $};
\draw (277.07,37.4) node [anchor=north west][inner sep=0.75pt]  [xscale=0.75,yscale=0.75]  {$\tau $};
\draw (245.05,36.5) node [anchor=north west][inner sep=0.75pt]  [xscale=0.75,yscale=0.75]  {$\tau $};
\draw (525.2,36.5) node [anchor=north west][inner sep=0.75pt]  [xscale=0.75,yscale=0.75]  {$\tau $};
\draw (442.15,36.5) node [anchor=north west][inner sep=0.75pt]  [xscale=0.75,yscale=0.75]  {$\tau $};

\end{tikzpicture}
\caption{An alternating $\tau$-path between $a$ and $b$}  \label{fig:tau-path}
\end{figure}
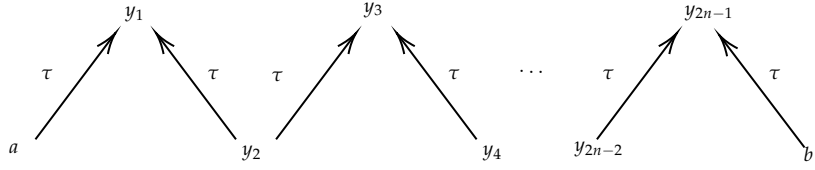

\begin{lemma}
  \label{Lem:connections}
  For all distinct $a, b\in M$ and any $2$-type $\tau$, there is $k \in \N$ such that for any finite $A \sub M$, there is an alternating $\tau$-path of length $k$ from $a$ to $b$ whose interior avoids $A$. In particular, there are infinitely many $\tau$-paths between $a$ and $b$ of length $k$ with pairwise disjoint interiors.
\end{lemma}
\begin{proof}
Let us first prove that there is an alternating $\tau$-path between $a$ and $b$.
Write $c \sim_\tau d$ if there is an alternating $\tau$-path between $c$ and $d$ or $c = d$. We check that $\sim_\tau$ is an equivalence relation. Symmetry and reflexivity are  clear from the definition. To check transitivity, consider a $\tau$-path $p_1$ from $c_0$ to $c_1$ and a $\tau$-path $p_2$ from $c_1$ to $c_2$ and suppose that $c_0$, $c_1$, and $c_2$ are distinct. By the no algebraicity assumption, there exists an element $g_1 \in G_{\set{c_0, c_1}}$ such that $c_2 \notin g_1 \cdot p_1$. Similarly, there exists $g_2 \in G_{\set{c_1, c_2}}$ such that $g_2 \cdot (p_2 \sminus \set{c_1, c_2}) \cap (g_1 \cdot p_1) = \emptyset$. Now the concatenation of $g_1 \cdot p_1$ and $g_2 \cdot p_2$ witnesses that $c_0 \sim_\tau c_2$.

By transitivity, there is $c \in M$ such that $\tp ac = \tau$. By the no algebraicity assumption, the orbit $G_c \cdot a$ is infinite, so the $\sim_\tau$-class of $a$ is infinite. By transitivity and weak elimination of imaginaries, it follows that the $\sim_\tau$-class of $a$ is all of $M$, so there is an alternating $\tau$-path between $a$ and $b$.

Now fix some alternating $\tau$-path $p$ between $a$ and $b$ and let $k$ be the length of $p$. By the no algebraicity assumption, there is $g \in G_{ab}$ that moves $p$ to a path whose interior is disjoint from $A$.
\end{proof}

Denote by $\lambda^\tau$ the distribution of $\eta^\tau_a$; this is a probability measure on $[0, 1]$ and by Lemma~\ref{Lem:i.i.d.}, it does not depend on $a$.
\begin{lemma}
  \label{Lem:atomless}
  Suppose that the measure $\lambda^\tau$ is non-atomic. Then for all $a, b \in M$, we have that, almost surely,
  \begin{equation*}
    a < b \iff \eta^\tau_a < \eta^\tau_b.
  \end{equation*}
\end{lemma}
\begin{proof}
First, we suppose that $D_\tau (a) \cap D_\tau (b) \neq \emptyset$. The contrapositive of Lemma~\ref{Lem:Inters} gives us that in that case,
\begin{equation}
    \label{eq:l:infinval2}
\eta_a^\tau < \eta_b^\tau \implies a < b.
\end{equation}
Next we consider the general case. Suppose that there exists an alternating $\tau$-path $y_0, \ldots, y_{2n}$ from $a$ to $b$ such that
\begin{equation}
  \label{eq:l:infinval1}
  \eta^\tau_{y_0} < \eta^\tau_{y_2} < \cdots < \eta^\tau_{y_{2n}}.
\end{equation}
Then for all $i$, $D_\tau(y_{2i}) \cap D_\tau(y_{2i+2}) \neq \emptyset$, so by the above observation, we obtain that $a = y_0 < y_2 < \cdots < y_{2n} = b$. Now condition on $\eta^\tau_a, \eta^\tau_b$ and suppose that $\eta^\tau_a < \eta^\tau_b$. As the $\eta^\tau_c$ are i.i.d. with non-atomic distribution, for a fixed $\tau$-path $(y_0, \ldots, y_{2n})$ between $a$ and $b$, \eqref{eq:l:infinval1} holds with positive probability that only depends on $n$. By Lemma~\ref{Lem:connections}, there exist infinitely many $\tau$-paths of the same length between $a$ and $b$ with disjoint interiors and whether \eqref{eq:l:infinval1} holds for each of them are independent events with the same probability. Thus almost surely at least one of them happens and we conclude that \eqref{eq:l:infinval2} holds for all $a, b$. For the reverse implication, it suffices to notice that $\Pr(\eta^\tau_a = \eta^\tau_b) = 0$.
\end{proof}

Lemma~\ref{Lem:atomless} allows us to conclude in the case where $\lambda^\tau$ is non-atomic.
\begin{lemma}
  \label{l:non-atomic}
  Suppose that for some type $\tau$, the distribution $\lambda^\tau$ is non-atomic. Then $\mu = \muu$.
\end{lemma}
\begin{proof}
  Define $\rho \colon \LO(M) \to [0, 1]^M$ by $\rho(x)(a) = \eta^\tau_a(x)$ ($\rho$ is defined $\mu$-a.e.). By Lemma~\ref{Lem:atomless}, $\pi \circ \rho = \id$ $\mu$-a.e. ($\pi$ is defined by \eqref{eq:def-pi}). By Lemma~\ref{Lem:i.i.d.}, $\rho_* \mu = (\lambda^\tau)^M$. Applying $\pi$ to both sides, we obtain that
  \begin{equation*}
    \mu = \pi_* \rho_* \mu = \pi_*(\lambda^\tau)^M = \muu. \qedhere
  \end{equation*}
\end{proof}

Now we are left with the case where the distribution $\lambda^\tau$ has atoms for all $2$-types $\tau$ and we will eventually conclude that there is a $G$-invariant linear order on $M$.

From here on, as we will deal with several measures simultaneously, we will incorporate the measure in our notation. If $\mu$ is an ergodic measure on $\LO(M)$, $\tau$ is a $2$-type, and $p \in [0, 1]$ is an atom for the distribution $\lambda^\tau$, we define a new measure $\nu_{\mu, \tau, p}$ on basic clopen sets by
\begin{equation}
  \label{eq:def:nu}
  \Pr_{\nu_{\mu, \tau, p}}(a_0 < \cdots < a_{k-1}) = \Pr_\mu(a_0 < \cdots < a_{k-1} \mid \eta_{a_0}^\tau = \cdots = \eta_{a_{k-1}}^\tau = p),
\end{equation}
where $a_0, \ldots, a_{k-1}$ are pairwise distinct elements of $M$. We note that as by Theorem~\ref{th:general-independence},
\begin{equation*}
  a_0 < \cdots < a_{k-1}, \eta_{a_0}, \ldots, \eta_{a_{k-1}} \indep \eta_{b_0}, \ldots, \eta_{b_{m-1}}
\end{equation*}
for any $\set{b_0, \ldots, b_{m-1}} \cap \set{a_0, \ldots, a_{k-1}} = \emptyset$, we can condition additionally on $\eta_{b_0}^\tau = \cdots = \eta_{b_{m-1}}^\tau = p$ on the right-hand side of \eqref{eq:def:nu} without changing the result.

For the next lemma, we will need the following well-known general ergodicity criterion.
\begin{prop}
  \label{p:indep-ergodicity}
  Let $H$ be a group and let $H \actson (X, \mu)$ be a measure-preserving action. Suppose that the collection
  \begin{equation*}
    \set{A \in \MALG(X, \mu) : \exists h \in H \ h \cdot A \indep A}
  \end{equation*}
  is dense in $\MALG(X, \mu)$. Then the action $H \actson X$ is ergodic.
\end{prop}
\begin{proof}
  Suppose that $B \in \MALG(X, \mu)$ is $H$-invariant. For every $\eps > 0$, there exist $A \in \MALG(X, \mu)$ and $h \in H$ such  that $\mu(A \sdiff B) < \eps$ and $h \cdot A \indep A$. We have that
  \begin{equation*}
    2(\mu(A) - \mu(A)^2) = \mu(A \sdiff h \cdot A) \leq \mu(B \sdiff h \cdot B) + 2\eps = 2\eps.
  \end{equation*}
  Taking a limit as $\eps \to 0$ yields that $\mu(B) - \mu(B)^2 = 0$, so that $\mu(B) = 0$ or $1$.
\end{proof}
\begin{lemma}
  \label{Lem:measures}
  Let $\mu$ be a $G$-invariant, ergodic measure on $\LO(M)$, $\tau$ be a $2$-type, and $p$ be an atom for $\lambda^\tau$. Then $\nu_{\mu, \tau, p}$ extends to a $G$-invariant, ergodic measure on $\LO(M)$.
\end{lemma}
\begin{proof}
  For brevity, write $\nu = \nu_{\mu, \tau, p}$. To define $\nu$ on a general clopen set $U$, we represent it as a disjoint union of basic clopen sets and use \eqref{eq:def:nu}. It follows from the remark after \eqref{eq:def:nu} that this is well-defined and gives rise to a finitely additive measure on the Boolean algebra of clopen subsets of $\LO(M)$.  Now it follows from the Carathéodory extension theorem that $\nu$ extends to a Borel measure on $\LO(M)$.

  The $G$-invariance of $\nu$ follows from \eqref{eq:def:nu} and the $G$-invariance of $\mu$. Finally, to check ergodicity, we will use Proposition~\ref{p:indep-ergodicity}. If $A = \set{a_0, \ldots, a_{k-1}} \sub M$, let $\cG_A$ be the Boolean algebra generated by the events $\set[\big]{\set{a_i <_x a_j} : i, j < k}$ and note that by regularity of $\nu$, $\bigcup_A \cG_A$ is dense in the measure algebra $\MALG(\LO(M), \nu)$. Now fix $A$ and use the no algebraicity assumption to find $g \in G$ such that $g \cdot A \cap A = \emptyset$. In order to apply Proposition~\ref{p:indep-ergodicity}, we will check that $g \cdot \cG_A \indep_\nu \cG_A$. It suffices to see that for any two permutations $(i_0, \ldots, i_{k-1})$ and $(j_0, \ldots, j_{k-1})$ of $(0, \ldots, k-1)$, the events $E_1 = \set{a_{i_0} <_x \cdots <_x a_{i_{k-1}}}$ and $E_2 = \set{g^{-1} \cdot a_{j_0} <_x \cdots <_x g^{-1} \cdot a_{j_{k-1}}}$ are $\nu$-independent. Let $C_1 = \set{\eta_{a_0}^\tau = \cdots = \eta_{a_{k-1}}^\tau = p}$ and $C_2 = \set{\eta_{g^{-1} \cdot a_0}^\tau = \cdots = \eta_{g^{-1} \cdot a_{k-1}}^\tau = p}$ (here the random variables are computed using $\mu$). It follows from Theorem~\ref{th:general-independence} that Boolean combinations of $E_1, C_1$ are $\mu$-independent from Boolean combinations of $E_2, C_2$. Using this, we have:
  \begin{equation*}
    \begin{split}
      \Pr_\nu(E_1 \cap E_2) &= \Pr_\mu(E_1 \cap E_2 \mid C_1 \cap C_2) \\
        &= \Pr_\mu(E_1 \mid C_1) \Pr_\mu(E_2 \mid C_2) = \Pr_\nu(E_1) \Pr_\nu(E_2).
    \end{split}
  \end{equation*}
  This concludes the proof.
\end{proof}

If $\tau$ is a $2$-type, say that a measure $\mu$ on $\LO(M)$ \df{respects $\tau$} if for all $a, b, c \in M$ such that $\tp ac = \tp bc = \tau$ and $\mu$-a.e. $x \in \LO(M)$, $c$  is not between $a$ and $b$ in the order $<_x$. Note that the uniform measure $\muu$ does not respect any type $\tau$.

\begin{lemma}
  \label{l:respecting-types}
  Let $\mu$ be a $G$-invariant, ergodic measure on $\LO(M)$, $\tau$ be a $2$-type, and $p$ be an atom for $\lambda^\tau$. Let $\nu = \nu_{\mu, \tau, p}$. Then the following hold:
  \begin{enumerate}
  \item \label{i:lrt:1} $\nu$ respects $\tau$;
  \item \label{i:lrt:2} If $\tau'$ is a $2$-type and $\mu$ respects $\tau'$, then $\nu$ respects $\tau'$.
  \end{enumerate}
\end{lemma}
\begin{proof}
  \ref{i:lrt:1} Let $a, b, c \in M$ be such that $\tp ac = \tp bc = \tau$. Using Lemma~\ref{Lem:Inters}, we have that
  \begin{equation*}
    \begin{split}
      \Pr_\nu(a < c < b) = \Pr_\mu(a < c < b \mid \eta^\tau_a = \eta^\tau_b = \eta^\tau_c = p) \\
      \leq \frac{\Pr_\mu(a < c < b \And \eta^\tau_a = \eta^\tau_b)}{\Pr_\mu(\eta^\tau_a = \eta^\tau_b = \eta^\tau_c = p)} = 0.
    \end{split}
  \end{equation*}
  We obtain similarly that $\Pr_\nu(b < c < a) = 0$.

  \ref{i:lrt:2} This is clear from the definition.
\end{proof}

\begin{lemma}
  \label{l:Dirac}
  Suppose that $\mu$ is a $G$-invariant, ergodic measure on $\LO(M)$ which respects all $2$-types. Then $\mu$ is a Dirac measure.
\end{lemma}
\begin{proof}
  We will prove that the order between two elements $a, b \in M$ is almost surely determined by $\tp ab$. More formally, we will show that for all $a \neq b$, we have that a.s.,
  \begin{equation*}
    \tp ab = \tp a'b' \implies \big(a < b \iff a' < b' \big).
  \end{equation*}
  What the hypothesis gives us is that for all $c, d, e \in M$, a.s.,
  \begin{equation}
    \label{eq:lD:1}
    \tp ce = \tp de \implies \big(c < e \iff d < e \big).
  \end{equation}
  Let $\tau = \tp ab = \tp a'b'$ and use Lemma~\ref{Lem:connections} to construct a $\tau$-path $a = y_0, \ldots, y_{2n} = a'$ from $a$ to $a'$ whose interior avoids $b$ and $b'$. Applying \eqref{eq:lD:1} consecutively, we obtain that:
  \begin{equation*}
    \begin{split}
      a < b &\iff a < y_1 \iff y_2 < y_1 \iff y_2 < y_3 \iff \cdots \\
      &\iff y_{2n-1} < a' \iff a' < b',
    \end{split}
  \end{equation*}
  which concludes the proof.
\end{proof}

Now we can complete the proof of the theorem.
\begin{proof}[Proof of Theorem~\ref{Thm:Princip}]
  Let $\mu$ be a $G$-invariant, ergodic measure on $\LO(M)$. Enumerate all $2$-types as $\tau_0, \ldots, \tau_{n-1}$. If $\lambda^{\tau_0}_\mu$ is non-atomic, then we can apply Lemma~\ref{l:non-atomic} and conclude that $\mu = \muu$. Otherwise, we construct a sequence of invariant, ergodic measures $\mu_0, \ldots, \mu_n$ such that for all $i < n$, $\mu_i$ respects $\tau_0, \ldots, \tau_{i-1}$ and $\lambda^{\tau_i}_{\mu_i}$ has an atom. Set $\mu_0 = \mu$ and suppose that $\mu_i$ is already constructed. Set $\mu_{i+1} = \nu_{\mu_i, \tau_i, p_i}$, where $p_i$ is some atom for $\lambda_{\mu_i}^{\tau_i}$. By Lemma~\ref{l:respecting-types}, $\mu_{i+1}$ respects $\tau_0, \ldots, \tau_i$. Moreover, $\lambda_{\mu_{i+1}}^{\tau_{i+1}}$ must have an atom: otherwise, by Lemma~\ref{l:non-atomic}, $\mu_{i+1} = \muu$, which is not possible because $\muu$ has full support and $\mu_{i+1}$ does not (as $\mu_{i+1}$ respects $\tau_i$). Finally, apply Lemma~\ref{l:Dirac} to conclude that $\mu_n$ is a Dirac measure, which, by invariance, implies that the action $G \actson \LO(M)$ has a fixed point.

  Thus we have proved that either $\muu$ is the unique ergodic, invariant measure on $\LO(M)$ or there is a fixed point for the action. However, as convex combinations of ergodic measures are dense in the space of all invariant measures (see, e.g., \cite{Phelps2001}*{Section~12}), this implies that in that case, $\muu$ is indeed the unique invariant measure.
\end{proof}

\begin{proof}[Proof of Corollary~\ref{c:intro:amenability}]
  Let $Z \sub \LO(M)$ be any minimal subsystem. By the hypothesis, $Z$ is not a point and it is a proper subset of $\LO(M)$. If $G$ is amenable, then there must be a $G$-invariant measure supported on $Z$, which contradicts Theorem~\ref{Thm:Princip}.
\end{proof}

\begin{proof}[Proof of Corollary~\ref{c:intro:Hrushovski}]
  By \cite{Kechris2007a}*{Proposition~6.4}, the Hrushovski property implies that there are compact subgroups $K_0 \leq K_1 \leq \cdots$ of $G$ with $\bigcup_n K_n$ dense in $G$. In particular, $G$ is amenable.

  If $K \leq G$ is any compact subgroup, then the orbits of $K$ on $M$ are finite. If $M$ admits a $G$-invariant linear order, then the $K$-orbits must be trivial, so $K$ is trivial. We conclude that there is no $G$-invariant linear order on $M$, so, by Corollary~\ref{c:intro:amenability}, the action $G \actson \LO(M)$ must be minimal. This implies that $M$ has the ordering property (see \cite{NguyenVanThe2013}*{Theorem~4}).
\end{proof}


\section{Examples and other invariant measures}
\label{sec:examples}

\subsection{Examples}
\label{sec:examples-1}

We briefly discuss some examples that show that the assumptions of Theorem~\ref{th:intro:deFinetti} and Theorem~\ref{th:intro:main} are mostly necessary. This section requires more familiarity with Fraïssé theory than the rest of the paper.

\subsubsection{Transitivity} Let $\cL$ be a language with two unary predicates $P$ and $Q$ and consider the age consisting of all $\cL$-structures for which $P \cap Q = \emptyset$ and every point satisfies either $P$ or $Q$. Let $M$ be its Fraïssé limit. Then one can randomly order $M$ as follows. Let $(\xi_a : a \in M)$ be uniform, i.i.d. on $[0, 1]$ and define an $\Aut(M)$-invariant random order $<$ on $M$ by declaring all elements of $P$ to be smaller than all elements of $Q$ and $a < b \iff \xi_a < \xi_b$ if $a$ and $b$ both belong to $P$ or to $Q$. This shows that the transitivity assumption in Theorem~\ref{th:intro:main} is necessary.
  
\subsubsection{No algebraicity} Let $V$ be the countable-dimensional vector space over $\bF_2$, the field with two elements. Let $V^*$ be its dual: the space of linear maps from $V$ to $\bF_2$. $V^*$ embeds as a subspace of $\bF_2^V$ and, being a compact group, has a Haar measure which is invariant under the action of $\Aut(V)$. This gives an invariant, ergodic measure on $\bF_2^V$ which is not a product measure and shows that one cannot omit the no algebraicity assumption in Theorem~\ref{th:intro:deFinetti}.

The same example also shows that this assumption cannot be omitted in Theorem~\ref{th:intro:main}. The universal minimal flow of $\Aut(V)$ is a proper subspace of $\LO(V)$ (see \cite{Kechris2005}*{Theorem~8.2}) and it carries a (unique) invariant measure \cite{Angel2014}*{Section~10}. This measure can be obtained as a factor of the measure on $\bF_2^V$ constructed above.

\subsubsection{Weak elimination of imaginaries} In the presence of $\Aut(M)$-invariant equivalence relations on $M$, it is easy to construct distributions for $(\xi_a : a \in M)$ for which the random variables are not independent. One can, for example, toss a coin for each equivalence class and set $\xi_a = 0$ or $1$ depending whether the coin toss for the class of $a$ resulted in heads or tails.

In view of Remark~\ref{rem:dF-weaker-assumption}, it is more interesting to ask whether the weak elimination of imaginaries assumption can be replaced just by requiring \df{primitivity} of the action $\Aut(M) \actson M$, that is, the absence of invariant equivalence relations on $M$. (This would also have the advantage of being much easier to check.) It turns out that the answer is negative, as the following example shows.

Let the signature $\cL$ consist of two unary relations $S_0$ and $S_1$ and a binary relation $R$. We consider the class $\cA$ of finite bipartite graphs viewed as $\cL$-structures, where the two parts of the graphs are labeled by $S_0$ and $S_1$ and $R$ is the edge relation. For a point $a$, we denote by $R(a)$ the set of $R$-neighbors of $a$. For elements of $\cA$, we require furthermore that the degree of every point in $S_0$ is $2$ and that $|R(a) \cap R(b)| \leq 1$ for all $a \neq b$ in $S_0$. It is easy to check that this is an amalgamation class; let $N$ be the Fraïssé limit. Denote $M = \set{a \in N : S_0(a)}$ and $P = \set{a \in N : S_1(a)}$. The class $\cA$ is not hereditary, so $N$ is not fully homogeneous but we do have homogeneity for algebraically closed, finite substructures of $N$. A finite substructure $A \sub N$ is algebraically closed iff for every $a \in A \cap M$, the degree of $a$ (calculated in $A$) is $2$ (that is, $A \in \cA$).

Now consider $M$ as a structure on its own (in a different signature) with relations given by the traces of all definable relations on $N$. As $N$ is $\aleph_0$-categorical, $M$ is too. Using the homogeneity of $N$, it is easy to check that $M$ has no algebraicity and that the action $\Aut(M) \actson M$ is primitive. Indeed, the action of $\Aut(M)$ on pairs of distinct elements of $M$ has exactly two orbits: $\set{(a, b) : |R(a) \cap R(b)| = 1}$ and $\set{(a, b) : R(a) \cap R(b) = \emptyset}$ and none of them is an equivalence relation.

There is a homomorphism $\Aut(N) \to \Aut(M)$ given by the natural action of $\Aut(N)$ on $M$. As the elements of $P$ can be recovered as imaginary elements of $M$, it turns out that this homomorphism is an isomorphism. With all of this in mind, it is easy to construct non-independent, $\Aut(M)$-invariant distributions of random variables $(\xi_a : a \in M)$. For example, we can start with i.i.d. $(\eta_b : b \in P)$ uniform in $[0, 1]$ and define
\begin{equation*}
  \xi_a = \min \set{\eta_b : b \in P, a \eqrel{R} b}.
\end{equation*}
This also allows to construct non-uniform, invariant measures on $\LO(M)$: just define a random order $<$ on $M$ as usual by $a < b \iff \xi_a < \xi_b$.

\subsubsection*{$\aleph_0$-categoricity.} We do not know whether $\aleph_0$-categoricity is necessary in either Theorem~\ref{th:intro:deFinetti} or Theorem~\ref{th:intro:main}, although it is crucial for our proofs. In the absence of $\aleph_0$-categoricity, however,  the other assumptions may need tweaking as the correspondence between model theory and permutation groups breaks down.

\subsection{Other invariant measures on $\LO$}
\label{sec:other-invar-meas}

One may ask, under the assumptions of Theorem~\ref{Thm:Princip}, what other ergodic, invariant measures there are on $\LO(M)$ apart from the uniform measure and fixed points. A slight variation of the method we used to construct $\muu$ yields the following. Let $\lambda$ be a probability measure on $[0, 1]$ and let $S = \set{z \in [0, 1] : \lambda(\set{z}) > 0}$ be the set of its atoms (it can be finite or countable). Let $F \sub \LO(M)$ be the set of $G$-fixed points (which, by Theorem~\ref{Thm:Princip}, has to be non-empty if we want to construct anything interesting) and finally, let $f \colon S \to F$ be an arbitrary function. Note that $\aleph_0$-categoricity of $M$ implies that $F$ is finite. Let $\pi \colon [0, 1]^M \to \LO(M)$ be defined ($\lambda^M$-a.e.) by
\begin{equation*}
  a <_{\pi(z)} b \iff z(a) < z(b) \Or \big( z(a) = z(b) \And a <_{f(z(a))} b \big).
\end{equation*}
Then $\pi_*(\lambda^M)$ is an invariant, ergodic measure on $\LO(M)$.

For many structures $M$, the methods we developed for the proof of Theorem~\ref{Thm:Princip} can be used to show that all ergodic, invariant measures on $\LO(M)$ can be obtained as above; however, in the presence of definable cuts, more complicated constructions are possible. We just give one example.

Consider the language $\cL = \set{<, f}$, where $<$ is a binary relation and $f$ is a unary function. Let $\cA$ be the age consisting of all finite $\cL$-structures where $<$ is interpreted as a linear order and $f$ is an involution without fixed points. It is easy to check that $\cA$ amalgamates; let $N$ be its Fraïssé limit. As for every $n$, the structure generated by $n$ points is of size at most $2n$ and there are only finitely many structures of any given finite size, $N$ is $\aleph_0$-categorical. Let $M = \set{a \in N : f(a) < a}$ and $M' = \set{a \in N : f(a) > a}$. It follows from homogeneity that $M$ and $M'$ are the two orbits of the action $\Aut(N) \actson N$. Now consider $M$ as a structure on its own with relations defined as the traces of definable relations from $N$ (in particular, the relations $a < b$, $f(a) < b$, $f(a) < f(b)$ for $a, b \in M$ are definable in the structure $M$). From a permutation group perspective, we can consider the homomorphism $\pi \colon \Aut(N) \to \Sym(M)$ given by the natural action $\Aut(N) \actson M$ and then $\Aut(M) = \cl{\pi(\Aut(N))}$ (this is because $\Aut(M)$ and $\Aut(N)$ have the same orbits on $M^k$ for every $k$). It follows from the homogeneity of $N$ that $M$ is transitive, $\aleph_0$-categorical, and has no algebraicity. (The algebraic closure operator in $N$ is given by $\acl(A) = A \cup f(A)$.) One can also verify weak elimination of imaginaries, for example using the criterion from \cite{Rideau2019}*{Proposition~10.1}.

We can construct an invariant measure on $\LO(M)$ as follows. Let $(\eta_a)_{a \in M}$ be a collection of i.i.d., Bernoulli, $\set{0, 1}$-valued random variables, where each of the two values is taken with probability $1/2$, and define a random order $\prec$ on $M$ by
\begin{equation*}
  a \prec b \iff f^{\eta_a}(a) < f^{\eta_b}(b),
\end{equation*}
where $f^0 = \id$ and $f^1 = f$. This random order is different from the ones considered above.

\bibliography{meas-LO}
\end{document}